\documentclass[12pt,a4paper]{amsart}
\usepackage{amsmath}
\usepackage{hyperref}
\usepackage[all]{xy}
\usepackage{amsxtra}
\usepackage[usenames]{color}
\usepackage{youngtab}
\usepackage{amscd}
\usepackage{amsthm}
\usepackage{amsfonts}
\usepackage{amssymb}
\usepackage{mathrsfs}
\usepackage{mathdots}
\usepackage{mathtools}
\usepackage{enumerate}
\usepackage{extarrows}
\usepackage{multirow}
\usepackage{epsfig,float,graphicx,verbatim}
\usepackage[T1]{fontenc}
\usepackage{stackrel}
\usepackage{MnSymbol}%
\usepackage{wasysym}%
\usepackage{babel}
\baselineskip 18pt \textwidth 16cm  \theoremstyle{plain}

\newtheorem*{theorem*}{Theorem}
\newtheorem*{remark*}{Remark}
\newtheorem*{example*}{Example}
\newtheorem{lemma}{Lemma}[section]

\newtheorem{theorem}[lemma]{Theorem}
\newtheorem{definition}[lemma]{Definition}

\newtheorem{conjecture}{Conjecture}
\newtheorem*{conjecture*}{Conjecture}

\newtheorem{thm}[lemma]{Theorem}
\newtheorem{prop}[lemma]{Proposition}
\newtheorem{lem}[lemma]{Lemma}
\newtheorem{defn}[lemma]{Definition}

\newtheorem{cor}[lemma]{Corollary}

\newtheorem{rem}[lemma]{Remark}

\newtheorem{introtheorem}{Theorem}

\newcommand{\s}{\sigma}

\makeatletter
\def\keywordsa{\xdef\@thefnmark{}\@footnotetext}
\makeatother

\makeatletter

\theoremstyle{remark}
\makeatother

\providecommand{\claimname}{Claim}

\title[Uneven Shalika models for $\GL_{n+m}$]{Uneven Shalika models for general linear groups}
\author{Itay Naor}
\date{\today}

\begin{document}
\global\long\def\C#1{\mathbb{C}^{#1}}%
\global\long\def\P#1{\mathbb{P}^{#1}}%
\global\long\def\Pr{\mathbb{P}}%
\global\long\def\E#1{\mathbb{E}\left[#1\right]}%
\global\long\def\R#1{\mathbb{R}^{#1}}%

\global\long\def\lbar#1{\bar{\ell}_{#1}}%
\global\long\def\xbar{\bar{p}}%
\global\long\def\norm#1{\left\lVert #1\right\rVert }%

\global\long\def\Xhat{\hat{X}}%
\global\long\def\homcor#1#2#3#4{[#1:#2:#3:#4]}%
\global\long\def\isom{\mathbb{\cong}}%
\global\long\def\Perp{^{\perp}}%
\global\long\def\F{\mathbb{F}}%
\global\long\def\l{\ell}%
\global\long\def\a{\alpha}%
\global\long\def\d{\delta}%
\global\long\def\b{\beta}%
\global\long\def\s{\sigma}%
\global\long\def\lb{\lambda}%
\global\long\def\pinv{^{+}}%

\global\long\def\del{\partial}%
\global\long\def\ra{\rightarrow}%
\global\long\def\da{\downarrow}%
\global\long\def\mt{\mapsto}%
\global\long\def\r{\rho}%
\global\long\def\e{\varepsilon}%
\global\long\def\vphi{\varphi}%
\global\long\def\qb{\overline{\mathbb{Q}}}%
\global\long\def\lbd{\lambda}%
\global\long\def\x{\chi}%
\global\long\def\Om{\Omega}%
\global\long\def\lhat#1{\widehat{#1}}%
\global\long\def\xh{\widetilde{\x}}%
\global\long\def\cA{\mathcal{A}}%
\global\long\def\fa{\forall}%
\global\long\def\cH{\mathcal{H}}%
\global\long\def\ca#1{\mathcal{#1}}%
\global\long\def\la{\leftarrow}%
\newcommand{\reals}{\mathbb{R}}

\global\long\def\ast{\textasteriskcentered}
\newcommand{\LL}{\mathbb{L}}
\newcommand{\sign}{\mathrm{sign}}
\newcommand{\half}{\frac{1}{2}}
\newcommand{\argmin}[1]{\underset{#1}{\mathrm{argmin}}}
\newcommand{\argmax}[1]{\underset{#1}{\mathrm{argmax}}}
\newcommand{\summ}{\displaystyle \sum}
\newcommand{\intt}{\displaystyle\int}
\newcommand{\var}{\text{Var}}
\newcommand{\nchoosek}[2]{\left(\begin{array}{*{20}c}#1\\#2\end{array}\right)}

\newcommand{\sfO}{\mathsf{O}}
\newcommand{\ba}{\mathbf{a}}
\newcommand{\be}{\mathbf{e}}
\newcommand{\bx}{\mathbf{x}}
\newcommand{\bw}{\mathbf{w}}
\newcommand{\bg}{\mathbf{g}}
\newcommand{\bb}{\mathbf{b}}
\newcommand{\bu}{\mathbf{u}}
\newcommand{\bv}{\mathbf{v}}
\newcommand{\bz}{\mathbf{z}}
\newcommand{\br}{\mathbf{r}}
\newcommand{\bc}{\mathbf{c}}
\newcommand{\bd}{\mathbf{d}}
\newcommand{\bh}{\mathbf{h}}
\newcommand{\by}{\mathbf{y}}
\newcommand{\bn}{\mathbf{n}}
\newcommand{\bs}{\mathbf{s}}
\newcommand{\bq}{\mathbf{q}}
\newcommand{\bmu}{\boldsymbol{\mu}}
\newcommand{\balpha}{\boldsymbol{\alpha}}
\newcommand{\bbeta}{\boldsymbol{\beta}}
\newcommand{\btau}{\boldsymbol{\tau}}
\newcommand{\bxi}{\boldsymbol{\xi}}
\newcommand{\blambda}{\boldsymbol{\lambda}}
\newcommand{\bepsilon}{\boldsymbol{\epsilon}}
\newcommand{\bsigma}{\boldsymbol{\sigma}}
\newcommand{\btheta}{\boldsymbol{\theta}}
\newcommand{\bomega}{\boldsymbol{\omega}}
\newcommand{\Lcal}{\mathcal{L}}
\newcommand{\Ocal}{\mathcal{O}}
\newcommand{\Acal}{\mathcal{A}}
\newcommand{\Gcal}{\mathcal{G}}
\newcommand{\Ccal}{\mathcal{C}}
\newcommand{\Xcal}{\mathcal{X}}
\newcommand{\Jcal}{\mathcal{J}}
\newcommand{\Dcal}{\mathcal{D}}
\newcommand{\Fcal}{\mathcal{F}}
\newcommand{\Hcal}{\mathcal{H}}
\newcommand{\Rcal}{\mathcal{R}}
\newcommand{\Ncal}{\mathcal{N}}
\newcommand{\Scal}{\mathcal{S}}
\newcommand{\Pcal}{\mathcal{P}}
\newcommand{\Qcal}{\mathcal{Q}}
\newcommand{\Wcal}{\mathcal{W}}
\newcommand{\Ld}{\tilde{L}}
\newcommand{\uloss}{\ell^\star}
\newcommand{\loss}{\mathcal{L}}
\newcommand{\losst}{\ell_t}

\newcommand{\inner}[1]{\langle#1\rangle}
\renewcommand{\comment}[1]{\textcolor{red}{\textbf{#1}}}
\newcommand{\vol}{\texttt{Vol}}

\newcommand{\secref}[1]{Sec.~\ref{#1}}
\newcommand{\subsecref}[1]{Subsection~\ref{#1}}
\newcommand{\figref}[1]{Fig.~\ref{#1}}
\renewcommand{\eqref}[1]{Eq.~(\ref{#1})}
\newcommand{\lemref}[1]{Lemma~\ref{#1}}
\newcommand{\corollaryref}[1]{Corollary~\ref{#1}}
\newcommand{\thmref}[1]{Thm.~\ref{#1}}
\newcommand{\propref}[1]{Proposition~\ref{#1}}
\newcommand{\appref}[1]{Appendix~\ref{#1}}

\newcommand{\note}[1]{\textcolor{red}{{#1}}}

\global\long\def\linf#1{\underset{#1\rightarrow\infty}{\lim}}%
\global\long\def\li#1#2{\underset{#1\rightarrow#2}{\lim}}%
\global\long\def\spnr#1#2{\mathrm{span}(#1,...,#2)}%
\global\long\def\spn{\mathrm{span}}%
\global\long\def\im{\mathrm{Im}}%
\global\long\def\argmin#1{\underset{#1}{\text{argmin}}}%

\global\long\def\mins#1#2{\underset{#1}{\min}\left(#2\right)}%
\global\long\def\maxs#1#2{\underset{#1}{\max}\left(#2\right)}%
\global\long\def\summ#1{\underset{#1}{\sum}}%
\global\long\def\intg#1{\underset{#1}{\int}}%
\global\long\def\sumud#1#2{\stackrel[#1]{#2}{\sum}}%
\global\long\def\piud#1#2{\stackrel[#1]{#2}{\Pi}}%
\global\long\def\intud#1#2{\stackrel[#1]{#2}{\int}}%
\global\long\def\cups#1#2{\stackrel[#1]{#2}{\cup}}%
\global\long\def\dsum#1#2{\stackrel[#1]{#2}{\oplus}}%

\global\long\def\rom#1{\mathrm{#1}}%
\global\long\def\un#1#2{\underset{#2}{#1}}%

\global\long\def\crt#1#2{#1\times#2}%
\global\long\def\bb#1{\mathbb{#1}}%
\global\long\def\rahat#1{\overrightarrow{#1}}%
\global\long\def\cPr#1#2{\mathbb{P}(#1|#2)}%
\global\long\def\pder#1#2{\frac{\partial#1}{\partial#2}}%
\newcommand{\Ind}{\operatorname{Ind}}
\newcommand{\Irr}{\operatorname{Irr}}
\newcommand{\ind}{\operatorname{ind}}
\newcommand{\Hom}{\operatorname{Hom}}
\newcommand{\tr}{\operatorname{tr}}
\newcommand{\diag}{\operatorname{diag}}
\newcommand{\GL}{\operatorname{GL}}
\newcommand{\Mat}{\operatorname{Mat}}
\global\long\def\cal#1{\mathcal{#1}}
\begin{abstract}
    We define a generalization of Shalika models for $\GL_{n+m}(F)$ and prove that they are multiplicity-free, where $F$ is either a non-Archimedean local field or a finite field and $n,m$ are any natural numbers. In particular, we give a new proof for the case of $n=m$. We also show that the Bernstein-Zelevinsky product of an irreducible representation of $\GL_n(F)$ and the trivial representation of $\GL_m(F)$ is multiplicity-free. We relate the two results by a conjecture about twisted parabolic induction of Gelfand pairs.
\end{abstract}
\maketitle
\keywordsa{
}
\tableofcontents 
\section{Introduction}

\subsection{Motivation}

In representation theory we are often interested to know when induced representations are multiplicity-free. In other words, we want to know when a representation induced from
a representation of a subgroup contains at most one copy of any irreducible
representation.

Therefore we define:
\begin{definition}

Let $H$ be a subgroup
of a group $G$ and $\psi$ a character of $H$. We say that the
triple $(G,H,\psi)$ is a twisted Gelfand pair if $\Ind_{H}^{G}\psi$
is multiplicity-free. In the case of a trivial character we
say that $(G,H)$ is a Gelfand pair. 

\end{definition}
The interest in multiplicity-free representations lies in the fact
that every irreducible subrepresentation has a canonical embedding
in it. For this reason multiplicity free representations are also called
models. Gelfand pairs have many applications, among them in the study of representations of symmetric groups (for example in \cite{sym}) and number theory (e.g. \cite{NT}).

    Throughout this paper, $F$ is either a finite field or a non-Archimedean local field. Let $G_{n+m}=\GL_{n+m}(F)$
be the general linear group over $F$ for $n,m\in \bb{N}$.
A classical Shalika model is the representation $\Ind_{H_{n,n}}^{G_{2n}}\psi$
where:
\[
H_{n,n}=\left\{ \left[\begin{array}{cc}
g & u\\
 & g
\end{array}\right]\mid g\in \GL_{n},u\in \Mat_{n}\right\} 
\]
 and $\psi$ is a generic character of: 
\[
U_{n,n}=\left\{ \left[\begin{array}{cc}
I_{n} & u\\
 & I_{n}
\end{array}\right]\mid u\in \Mat_{n}\right\} 
\]
 extended trivially to $H_{n,n}$. Jacquet and Rallis have proved
in \cite{JacRal} that classical Shalika models are multiplicity-free for p-adic fields. That is, in this case
$(G_{2n},H_{n,n},\psi)$ is a twisted Gelfand pair.

We define a subgroup $H_{n,m}< G_{n+m}$ generalizing the Shalika subgroup $H_{n,n}< G_{2n}$ and investigate a generalization of Shalika models. We show that the induction of a generic character of $H_{n,m}$ to $G_{n+m}$ is multiplicity-free. We call this induced representation an \textit{uneven Shalika model}.

The main motivation for the second Gelfand pair studied in this work
is a conjecture about twisted parabolic induction of Gelfand pairs. To describe the conjecture we consider a reductive group $G$ and assume that $P=LU$ is a Levi decomposition of a parabolic subgroup of $G$, where $L$ and $U$ are the corresponding Levi and unipotent subgroups respectively. Inspired by parabolic induction of spherical pairs, it has been conjectured that:
\begin{conjecture}\label{con:dim}
(Aizenbud-Gourevitch-Sayag) Let $H_0$ be
a subgroup of $L$ and $\psi$ a character of $H_0U$ whose restriction to $U$ is generic. If $(L,H_0,\psi|_{H_0})$ is a twisted Gelfand pair then $(G,H_0U,\psi)$ is a twisted Gelfand pair. 
\end{conjecture}
Embedding $\GL_{n}\times \GL_{m}$
as a Levi subgroup in $\GL_{n+m}(F)$, we will see that the conjecture ties the triple $(G_{n+m},H_{n,m},\psi)$ to a pair in $\GL_{n}\times \GL_{m}$ which we describe precisely in the next subsection. We prove in this paper that this pair is a Gelfand pair as well.

\subsection{Main results}

In this work we prove the Gelfand property for two pairs inspired
by Shalika models. Let $G:=G_{n+m}$ and recall that $F$ is either a finite field or a non-Archimedean local field.

\subsubsection{$\textsc{Uneven Shalika models}$}
We define a generalization of Shalika models for
$m\neq n$. Assume without
loss of generality $n\leq m$. Let $P_{n,n,m-n}$ be the standard parabolic subgroup
corresponding to the partition $\{n,n,m-n\}$, and let $P_{n,n,m-n}=LU$ be its Levi
decomposition. Denote
\[
H_{0}:=\{\diag(g,g,h) \mid g\in \GL_{n}(F),h\in \GL_{m-n}(F)\}<L
\]
and define the Shalika subgroup to be $H=H_{n,m}:=H_{0}U<P_{n,n,m-n}$. 

Choose a non-trivial character $\psi_{0}$ of the additive group of
$F$. Let $\psi$ be any character of $H$ that satisfies
\[
\psi(\left[\begin{array}{ccc}
I_{n} & u & a\\
 & I_{n} & b\\
 &  & I_{m-n}
\end{array}\right]):=\psi_{0}(tr(u))
\]
for all $a,b\in \Mat_{n,m-n}$ and $u\in \Mat_{n}$. We call any such
character a twisted Shalika character. In the special
case of $\psi|_{H_{0}}\equiv1$ we call $\psi$ a regular (non-twisted)
Shalika character. 
\begin{introtheorem}
\label{thm:main}The pair $(G_{n+m},H_{n,m},\psi)$ is a twisted
Gelfand pair. Explicitly,
\[
\dim\Hom_{G}(\pi,\Ind_{H}^{G}\psi)\leq1
\]

for any $\pi\in \Irr(G_{n+m})$.
\end{introtheorem}

Note that in the case of $n=m$ we get the uniqueness of classical Shalika models. 

Our proof of theorem \ref{thm:main} is geometric in nature.
\subsubsection{$\textsc{A Gelfand pair in } \GL_{n}(F)\times \GL_{m+n}(F)$}
We now describe another Gelfand pair. 
We choose a standard parabolic subgroup corresponding to the partition $\{n,m\}$ of $n+m$: \[P=P_{n,m}:=\left\{ \left[\begin{array}{cc}
g_{1} & u\\
 & g_{2}
\end{array}\right] \mid g_{1}\in \GL_{n}(F),g_{2}\in \GL_{m}(F),u\in \Mat_{n,m}(F)\right\}< G \] 
and define: 
\[\Delta P=\Delta P_{n,m}:=\left\{ (g_{1},\left[\begin{array}{cc}
g_{1} & u\\
 & g_{2}
\end{array}\right])\in \GL_{n}(F) \times P_{n,m} \right\}. \]
and let $\chi$ be a character of $\Delta P$ which is trivial on $\{(I_n,\left[\begin{array}{cc}
I_n & u\\
 & I_m
\end{array}\right])\mid u\in \Mat_{n,m}\}.$
We will deduce the following theorem from Theorem \ref{thm:main}.
\begin{introtheorem}
\label{thm:summer}The pair $(\GL_n(F) \times G,\Delta P,\chi)$ is a Gelfand pair.
\end{introtheorem}
In the case of $\chi\equiv1$, the theorem can be formulated in an additional way.
We denote the Bernstein-Zelevinski product, introduced in \cite{BZ-prod}, by $\times$.
\begin{introtheorem}
\label{thm:erez-version}If $\pi$ is an irreducible representation
of $\GL_{n}(F)$ and $1_m$ is the trivial representation of $\GL_{m}(F)$ then $\pi\times1_m$ is multiplicity-free. 
\end{introtheorem}

\subsection{Related work}
Some special cases of the theorems are already known.

\subsubsection{$\textsc{Shalika models}$}
The uniqueness of classical Shalika models (the $m=n$ case) was first
proven for p-adic fields in \cite{JacRal} by deducing it from the uniqueness
of linear models. The uniqueness of Shalika models in the case of Archmidean local fields was proven in \cite{shalikaR} by embedding them in linear models as well. In \cite{Nien} Nien outlined a direct proof using a geometric technique
similar to the one in the current paper, but there
is a gap in Lemma 3.9. This is explained in Appendix \ref{app:nien}. This paper in particular fills this gap.

In the case $m=n$, the uniqueness property was proven for twisted Shalika models
in \cite{ChenSun} by embedding them in twisted linear models. 
The generalization to $n\neq m$, as well as the proof for finite fields, are novel.

Several generalizations of Shalika models to other groups have been studied. For example, the uniqueness of Shalika models for $SO(4n)$ was shown in \cite{JiQi}.

\subsubsection{$\textsc{A Gelfand pair in } \GL_{n}\times \GL_{m+n}$}
In the case of a non-Archmidean local field, Theorem \ref{thm:erez-version} is a special case of \cite{erez}. It is shown there that $\text{soc}(\pi\times\tau)$ is actually irreducible for any $\square$-irreducible $\tau$, and in particular for any irreducible unitary representation $\tau$. The proof in \cite{erez} uses the classification of representations of $\GL_{n}(F)$ for a non-Archimedean field. We use more elementary method to show that $\text{soc}(\pi\times 1)$ is multiplicity-free. Our proof works over finite fields as well.

Note that Theorem \ref{thm:summer} is weaker than the claim that $(\GL_{n+m},\GL_{n})$ is a Gelfand pair. This pair is indeed a Gelfand pair in the case of $n=1$ for any local field (see \cite{k1ch0}, \cite{k1chp} and \cite{k1arch}), but fails in the case of a finite field already for $n=m=1$.

Another result generalizing Theorems \ref{thm:summer} is true for $n=1,2$. Recall that $L'$ is the Levi part of $P_{n,m}$. In \cite{dimaiz} it is proven that the pair
$(L'\times G, \widetilde{\Delta P})$
is a Gelfand pair
where $\widetilde{\Delta P}=\{(X,Y)\in L'\times P \mid X$ is the Levi part of $Y\}$.

Yaron Brodsky proved a result similar to Theorems \ref{thm:summer} for permutation groups. He showed that $(S_{n+m}\times S_n, \Delta S)$ is a Gelfand pair, where $\Delta S=\{(\s,\s') \mid \s\in S_{n+m}\text{ which preserves $[n]$ and $\s'=\s|_{[n]}$}\}$. 

\subsection{Structure of this paper}
Section \ref{sec:pre} presents some necessary preliminaries. In Subsection \ref{sec:ded} we prove that Theorem \ref{thm:main} implies Theorem \ref{thm:summer} and the rest of Section \ref{sec:main} is devoted to the proof of Theorem \ref{thm:main}. In Subsection \ref{sec:dec} we prove several useful lemmas we use in Subsections \ref{sec:dense} and \ref{sec:red} to reduce the problem to a simpler problem presented in Subsection \ref{sec:genset}. Finally, in Subsection \ref{sec:simprob} we prove this similar problem.
\subsection{Acknowledgements}
I wish to thank Avraham Aizenbud, Tomer Novikov, Lior Silberberg and Oksana Yanshyna for helpful discussion and suggestions. I would like to express my gratitude to my advisor Dmitry Gourevitch.

The author was supported in part by the BSF grant N 2019724 and by a fellowship endowed by Mr. David Lopatie.
\section{Preliminaries}\label{sec:pre}

\subsection{Notations}
We will use the terminology introduced in \cite{BZ} by Bernstein and Zelevinsky regarding representations of $l$-groups.
\begin{itemize}
\item
Let $X$ be an $l$-space. Write $S(X)$ for for the space of compactly supported
locally constant $\bb{C}$-valued functions on $X$. These functions are called Schwartz functions. The linear functionals in the dual space $S^{*}(X)$
are called distributions.
\item For a representation $(\pi,G,V)$ of an $l$-group, the subset of smooth vectors (i.e. vectors in $V$ whose stabiliser is open) is denoted by $V^{\infty}$. A representation is called smooth if $V^\infty =V$. The representation is  admissible if for any open subgroup $H< G$
the subspace of $H$-invariant vectors $V^H$ is finite-dimensional.
\item If $(\pi,G,V)$ is a representation and $\chi$ is a character of
a subgroup $H$, we define the subspace of invariants $V^{H,\chi}:=\{v\in V \mid \pi(h)v=\chi(h)v\}$
and the space of coinvariants $V_{H,\x}:=V/\text{span}\{\pi(h)v-\chi(h)v \mid h\in H,v\in V\}$.
\item  Denote by $e_{i}$ and $E_{ij}$ elements of the standard bases
of $F^{N}$ and $\Mat_{N,M}$, the dimensions $M,N$ will be clear
from context. 
\item For a set $S\subset \bb R$ and scalars $\alpha,\beta\in\bb {R}$ we denote $\alpha S+\beta=\left\{ \alpha s+\beta \mid s\in S\right\}$. We denote $[n]:=\left\{1,...,n\right\}$ as well.
\item
Write
$$w_{k}=\left[\begin{array}{ccc}
0 &  & 1\\
 & \iddots\\
1 &  & 0
\end{array}\right]\in \GL_{k}.$$ 
\item
The permutation matrices in this paper are in row representation. I.e., for a permutation $\s\in S_n$ corresponds the matrix $Q\in \GL_n$ satisfying $Q_{i,\downarrow}=e_{\s(i)}$ for all $i$.
\item
For a permutation $p\in S_{n}$ and a set $S\subset[n]$,
we denote $p(S)=\{p(i)\mid i\in S\}$. 
\item Let $G_{x}$ and $[x]$ denote the stabiliser and orbit of a group $G$ acting on a space
$X$ at the point $x\in X$. We write $x_1\sim x_2$ if they are in the same orbit or coset.
\item
Given a matrix $A\in \Mat_{n}$ and a partition  $n=n_{1}+...+n_{k}$,
the partition of $A$ to blocks according to $\{n_{1},...,n_{k}\}$
is the partition where the block in the position $(i,j)$ has
$n_{i}$ rows and $n_{j}$ columns. Therefore, the standard parabolic subgroup corresponding to this partition, denoted by $P_{n_1,...,n_k}$, is the subgroup of upper triangular block matrices with diagonal blocks of dimensions $n_i\times n_i$.
\item
For a set $S$ we write $FG(S)$ for the free group on the elements of $S$.
\end{itemize}
\subsection{The Gelfand-Kazhdan criterion}\label{sec:GK}

In this section, G is a general group (not necessarily $\GL_{n+m}(F)$) and $H$ is a subgroup. 
As before, we say that a representation $(\pi,G,V)$ is multiplicity-free
if $\dim\Hom_{G}(\rho,\pi)\leq1$ for any $\rho\in\Irr(G)$. 

A map $\tau:G \to G$ is called an anti-involution if $\tau\circ\tau=\text{id}_{G}$ and
$\tau(gg')=\tau(g')\tau(g)$
for any $g,g'\in G$.

The Gelfand-Kazhdan criterion is an important tool for identifying
Gelfand pairs. We will state the versions of the criterion we will
need.
We first deal with the case of a finite group. Let $\psi$ be a character of $H$ and $\tau$ an anti-involution. Write $H':=\tau(H)$ and $\widetilde{H}:=H'\times H$. 
Define an action of $\widetilde{H}$ on $G$ by 
$
    (h_{1},h_{2}).g=h_{1}gh_{2}^{-1}
$ for $h_{1}\in H'$, $h_{2}\in H$ and $g\in G$.
\begin{definition}
We say that $[g]$ is $\psi$-admissible if $\psi(\tau(h_{1})h_{2}^{-1})=1$
for all $(h_{1},h_{2})\in \widetilde{H}_{g}$ and that $[g]$ is $\psi$-$\tau$-invariant
if there exist $(h_{1},h_{2})\in \widetilde{H}$ with $(h_{1},h_{2}).g=\tau(g)$
and $\psi(\tau(h_{1})h_{2}^{-1})=1$.
\end{definition}
Now we can formulate Gelfand's trick for finite groups, which is justified by the classical argument described in \cite{aiz}.
\begin{thm}\label{thm:GKfinite}
Let $G$ be a finite group and $H< G$ be a subgroup. If there exists an anti-involution $\tau$ of $G$ and a character $\psi$ of $H$ such that every $\psi$-admissible coset in $\tau(H) \backslash G / H$ is $\psi$-$\tau$-invariant, then $(G,H,\psi)$ is a twisted Gelfand pair.
\end{thm}

Now we move to the  non-Archimedean case. 
Assume that $G$ is an $l$-group, $H$ is a subgroup and $\tau$ is an anti-involution. Define $H'$ and $\widetilde{H}$ as before. 

We define the action of $\widetilde{H}$ on $G$ in the same way as in the finite case. On  $S(G)$ and $S^*(G)$ we define: 
\begin{equation*}
\begin{gathered}
((h_{1},h_{2}).f)(g)=\psi(\tau(h_{1})^{-1}h_{2})f(h_{1}^{-1}gh_{2}),\\
    ((h_{1},h_{2}).T)(f)=T((h_{1}^{-1},h_{2}^{-1}).f)
\end{gathered}
\end{equation*} 
for $h_{1}\in H'$, $h_{2}\in H$,$g\in G, f\in S(G)$ and $T\in S^*(G)$. We define $\psi$-admissibility and $\psi$-$\tau$-invariance in the same way as in the finite case.

We need the following result.
\begin{theorem}[Gelfand-Kazhdan Criterion] \label{thm:GK-orig}\cite[Theorem 4]{GK}
Let $\pi\in\Irr(G)$. If every distribution in $S^*(G)$ invariant to the action of $\widetilde{H}$ defined above is also $\tau$ invariant, then: 
\[\dim \Hom _G ( \pi,Ind_H^G\psi) \cdot\dim \Hom _G (\widetilde{\pi},Ind_{H'}^G(\psi\circ \tau))\leq 1.\]
\end{theorem}
This result is true for any $l$-group $G$. From now until the end of the section we assume $G=\GL_{n_1}(F)\times...\times \GL_{n_k}(F)$ where $F$ is a non-Archimedean field. We can deduce a stronger result in this case under certain conditions on the anti-involution $\tau$. For  $(g_1,...,g_k)\in G$ we define $(g_1,...,g_k)^t=(g_1^t,...,g_k^t)$.
\begin{theorem}[Application of Gelfand-Kazhdan Criterion]
\label{thm:GKcrit} Let $(\pi,V)$ be an irreducible representation of $G$. Assume that there exist some fixed $a,b\in G$ such that for all $g\in G$ and  $h'\in H'$ hold $\tau(g)=ag^ta^{-1}$, $b\tau(h'^{-1})b^{-1}\in H$ and $\psi(b\tau(h'^{-1})b^{-1})=\psi(\tau(h'))$. If every distribution invariant to the action of $\widetilde{H}$ defined above is also $\tau$ invariant, then: \[\dim \Hom _G ( \pi,Ind_H^G\psi)\leq 1.\]
\end{theorem}
\begin{proof}
The proof is standard and uses the fact that $\widetilde{\pi}\isom \pi\circ \kappa$  (see \cite[Theorem 2]{GK}).
\end{proof}
The following lemma is a useful tool for verification of the condition of the Gelfand-Kazhdan Criterion. Let $U$ be a unipotent subgroup of $G$ and denote $U':=\tau(U)$. We take any character of $\widetilde H$ which coincides with  $\psi$ on $U'\times U\cap \widetilde{H}$ and denote it by $\psi_u$. For this lemma we also need to assume that $H$ is an algebraic subgroup of $G$. We also define an automorphism $\iota$ of $G\times G$ by $\iota(g_1,g_2)=(\tau(g_2^{-1}),\tau(g_1^{-1}))$ and denote $\widetilde{\psi}(h_1,h_2):=\psi(\tau(h_{1})h_{2}^{-1})$.
\begin{lem}\label{lem:tauinv}
Assume that every $\psi_u$-admissible double coset in $H'\backslash G/H$
is $\psi$-$\tau$-invariant. Then every distribution $T\in S^*(G)$ which is $\widetilde{H}$ invariant (in relation to the action defined above) is also $\tau$ invariant.
\end{lem}
\begin{proof}
To get this result from \cite[Theorem 6.10]{BZ}, we have to show that:
\begin{enumerate}
    \item The action of $\widetilde{H}$ on $G$ is constructible (i.e. the graph of the action of $\widetilde H$ on $G$ is a finite union of locally closed subsets).
    \item For any $(h_1,h_2)\in \widetilde{H}$ there exist $(h'_1,h'_2)\in \widetilde{H}$ such that $(h_1,h_2).\tau(g)=\tau((h'_1,h'_2).g)$ for all $g \in G$.
    \item There exists $n_0\in\bb{N}$ and  $\tilde h\in \widetilde{H}$ such that $\tau^{n_0}$ induces the same action on $G$ as $\tilde h$.
    \item If $Y=H'\gamma H$ is an orbit and $T\in S^*(Y)$ is a non zero $\widetilde{H}$-invariant distribution on it, then $Y$ and $T$ are $\tau$ invariant.
\end{enumerate}
    We now verify the conditions. Firstly, (1) follows from the fact that the action is algebraic and \cite[Theorem A in 6.15]{BZ}. For (2) note that $\tilde h.\tau(g)=\tau(\iota(\tilde h).g)$ for any $\tilde h\in \widetilde{H},g\in G$. (3) is clear since $\tau^2=id$.
    
To show (4), note that $Y\isom \widetilde{H}/\widetilde{H}_\gamma$ (recall that $\widetilde{H}_\gamma$ is the stabilizer). 
We note that $\widetilde{\psi}\circ\iota(\tilde h)=\widetilde{\psi}(\tilde h)$ for any $\tilde h\in \widetilde{H},g\in G$.

Consider the space of Schwarz functions $S(Y)$ with the standard action. We have $ind_{\widetilde{H}_\gamma}^{\widetilde{H}}(1)\isom S(Y)$ by the definition of the compact induction functor $ind$. Therefore by Frobenius reciprocity:
\[
T\in \Hom_{\widetilde{H}} (S(Y), \widetilde{\psi})\isom \Hom_{\widetilde{H}_\gamma}(\Delta_{\widetilde{H}}/\Delta_{\widetilde{H}_\gamma},\widetilde{\psi})
\]
Since $T\neq 0$, we must have $\widetilde{\psi}=\Delta_{\widetilde{H}}/\Delta_{\widetilde{H}_\gamma}$ on ${\widetilde{H}_\gamma}$. Since algebraic characters are trivial on unipotent subgroup and modular characters of algebraic groups are a composition of an algebraic character and an absolute value, the character $\Delta_{\widetilde{H}}/\Delta_{\widetilde{H}_\gamma}$ is trivial on $U'\times U$. We deduce that $Y$ is $\psi_u$-admissible. By our assumption, there exists $\tilde h_0\in \widetilde H$ with $\tilde h_0.\gamma=\tau(\gamma)$
and $\widetilde{\psi}(\tilde h_0)=1$. Therefore, $Y$ is $\tau$-invariant. 

If $T$ is not $\tau$ invariant then $T_1:=T-\tau(T)$ is not zero and is $\widetilde{H}$-invariant as well. We note that $\tau(T_1)=-T_1$. Now, by \cite[6.12]{BZ} $T_1$ is proportional to the distribution
\[
T_0(f)=\int_{\widetilde{H}/\widetilde{H}_g}f(\tilde h.g)\widetilde{\psi}^{-1}(\tilde h)d\tilde h
\]
where the measure $d\tilde h$ is a left invariant Haar measure on $\widetilde{H}/\widetilde{H}_g$. Then:
\begin{align*}
-T_0(f)&=\tau T_0(f)=\int_{\widetilde{H}/\widetilde{H}_g}f(\tau(\tilde h.g)) \widetilde{\psi}^{-1}(\tilde h)d\tilde h
=\int_{\widetilde{H}/\widetilde{H}_g}f(\iota(\tilde h).\tau(g))\widetilde{\psi}^{-1}(\tilde h)d\tilde h\\
&=\int_{\widetilde{H}/\widetilde{H}_g}f(\iota(\tilde h)\tilde h_0.g)\widetilde{\psi}^{-1}(\tilde h)d\tilde h
=c\int_{\widetilde{H}/\widetilde{H}_g}f(\tilde k.g)\widetilde{\psi}^{-1}(\tilde k)\widetilde{\psi}(\tilde h_0)d\tilde k=cT_0(f)
\end{align*}
using a change variables $\tilde k=\iota(\tilde h) \tilde h_0$ and the fact that $\psi(\tilde h_0)=1$. The constant $c$ is the  odulus of the automorphism $\iota$, which is positive. Thus $T_0(f)=0$ for all $f\in S(G)$.

Therefore $T$ is $\tau$ invariant.
\end{proof}
Before we move to the proof, let us remark that we can replace the condition of $\psi$-admissbility in the finite case with $\psi_u$-admissbility, as the same argument still works.

\section{Proof of Theorem \ref{thm:main}}\label{sec:main}
Recall that we chose $G=GL_{n+m}$ and $$H=\left\{\left[\begin{array}{ccc}
g_1 & u & a\\
 & g_1 & b\\
 &  & g_2
\end{array}\right]\mid  g_1\in \GL_{n}, g_2\in \GL_{m-n}, a,b\in \Mat_{n,m-n}, u\in \Mat_{n}\right\}.$$ We choose an anti-involution:
\[
\tau(g)=\left(\begin{array}{cc}
w_{2n}\\
 & I_{m-n}
\end{array}\right)g^{t}\left(\begin{array}{cc}
w_{2n}\\
 & I_{m-n}
\end{array}\right)
\]
and consider the cosets $H^{'}\backslash G/ H$ where $H'=H_{n,m}':=\tau(H)<G$. Denote $\widetilde{H}:=H'\times H$. Recall that $\psi$ is a generic character of $H$. We define a character $\psi_u$ of $ H$ by setting $\psi_u=\psi$ on the unipotent subgroup corresponding to $P_{n,n,m-n}$ and $\psi_u|_{H_0}\equiv1$.

In the following section we will prove the following proposition.

\begin{prop}[Geometric statement]
\label{prop:adm-inv}Every $\psi_u$-admissible double coset in $H^{'}\backslash G/H$
is $\psi$-$\tau$-invariant. 
\end{prop}
From now on, whenever we say that a coset is admissible we mean that it is $\psi_u$-admissible unless otherwise specified. 

\begin{proof}[Proof of Theorem \ref{thm:main} assuming Proposition \ref{prop:adm-inv}]
The the conditions of Lemma \ref{lem:tauinv} are satisfied, so it only remains to verify that the anti-involution we chose satisfies the conditions in Theorem \ref{thm:GKcrit}. 
Choose $a=\diag(w_{2n}, I_{m-n})$
and $b=\diag(I_n,-I_n, I_{m-n}).$ 
A routine calculation shows that $\tau(g)=ag^ta^{-1}$ and in addition   $b\tau(h'^{-1})b^{-1}\in H$ and $\psi(b\tau(h'^{-1})b^{-1})=\psi\circ\tau(h')$ for all $g\in G, h'\in H'$.

Theorem \ref{thm:main} follows now from Lemma \ref{lem:tauinv} and Theorem  \ref{thm:GKcrit} (and Theorem \ref{thm:GKfinite} with the aforementioned altered Proposition \ref{prop:adm-inv}).
\end{proof}

The proof of the geometric statement consists of several steps. We divide the cosets in $H'\backslash G/H$ into two subsets. In Subsection \ref{sec:dense} we prove  directly for the cosets in the first subset that every admissible coset is $\psi$-$\tau$-invariant. The union of these cosets is dense in $G$. In order to solve the problem for the second subset, we define in Subsection \ref{sec:genset} a similar problem which implies the claim for the remaining cosets. 

\subsection{Deduction of Theorem \ref{thm:summer} from Theorem \ref{thm:main}}\label{sec:ded}
In this subsection we show how the uniqueness of uneven Shalika models implies Theorem \ref{thm:summer}. First, we note that $\Delta P_{n,m}$ can be embedded as $$H_1:=\left\{\diag(g_1,\left[\begin{matrix}
g_1 & u \\
 & g_2 
\end{matrix}\right])\right\}$$ in $G_{n+m}$. Let $P_{n,m}=L'U'$ be the Levi decomposition of the parabolic subgroup corresponding to $\{n,m\}$. Theorem \ref{thm:summer} states that $(L',H_1,\psi|_{H_1})$ is a Gelfand pair.

Assume to the contrary that $\dim \Hom_{L'}(\pi,\Ind_{H_1}^{L'}(\psi|_{H_1}))>1$ for some $\pi \in \Irr(L')$. Let $T_1,T_2$ be two linearly independent homomorphisms in $\Hom_{L'}(\pi,\Ind_{H_1}^{L'}(\psi|_{H_1}))$. Consider the map $T_1+T_2:\pi\oplus\pi\ra \Ind_{H_1}^{L'}(\psi|_{H_1})$ defined by $(v_1,v_2)\mt T_1(v_1)+T_2(v_2)$. 

Since the projections of the kernel to each of the coordinates are $L'$-invariant and $\pi$ is irreducible, they are either $0$ or $\pi$. We consider all the possible cases regarding the projections of $\ker (T_1+T_2)$. If both are $\pi$, we get an isomorphism that maps $v\in\pi$ to a $u\in\pi$ such that $T_1(v)+T_2(u)=0$. Using schur's lemma we deduce that $T_1$ and $T_2$ are linearly dependent, which contradicts our assumption. If exactly one of the projections is $\pi$ we get that one of the maps is zero, which is again a contradiction. Therefore both projections are 0 the map $T_1+T_2$ must be injective. We've shown that $\pi\oplus\pi$ is embedded in $\Ind_{H_1}^{L'}(\psi|_{H_1})$.

Tensoring with $\psi_u$ and applying the induction functor, we get that $\Ind_{L'U'}^G(\pi\otimes\psi_u)\oplus\Ind_{L'U'}^G(\pi\otimes\psi_u)$ is embedded in $\Ind_{L'U'}^G(\Ind_{H_1}^{L'}(\psi|_{H_1})\otimes\psi_u)=\Ind_{H}^G(\psi)$, showing that $\Ind_{H}^G(\psi)$ is not multiplicity free in contradiction to Theorem \ref{thm:main}.
\qed

\subsection{Decomposing representatives as a direct sum}\label{sec:dec}

In the reduction of Proposition \ref{prop:adm-inv} to the aforementioned similar problem, we
will need to decompose matrices and solve the problem separately for
each part. In order to do so, we define a notion of a direct sum of
matrices and give conditions regarding its compatibility with admissibility and $\psi$-$\tau$-invariance.

For any subset $S\subseteq[n]$ we will denote by $S_{i}$ the $i$-th
smallest element. Define an operator $$E_{S_{1},S_{2}}(A)=\summ{\mathclap{i\leq|S_{1}|,j\leq|S_{2}|}}A_{ij}E_{(S_1)_{i},(S_2)_{j}}$$
for $S_{1},S_{2}\subset[n]$ and $A\in \Mat_{|S_{1}|,|S_{2}|}$ (for clarification, $E_{(S_1)_{i},(S_2)_{j}}$ is the standard basis matrix with 1 at the row and column corresponding to the $i$-th smallest element in $S_1$ and the $j$-th smallest element in $S_2$, respectively).
\begin{lem}
\label{lem:Eprops}For any $S_{1},S_{2},S_{3}\subset[n]$ we can state multiplication rules for the operator $E$:
\[
E_{S_{1},S_{2}}(A)E_{S_{2},S_{3}}(B)=E_{S_{1},S_{3}}(AB)
\]
\[
E_{S_{1},S_{2}}(A)E_{S_{2}^{C},S_{3}}(B)=0.
\]In addition,
\[
E_{S_{1},S_{2}}(A)^{t}=E_{S_{2},S_{1}}(A^{t}).
\]

and if $A,B$ are invertible:
\[
(E_{S_{1},S_{1}}(A)+E_{S_{1}^{C},S_{1}^{C}}(B))^{-1}=E_{S_{1},S_{1}}(A^{-1})+E_{S_{1}^{C},S_{1}^{C}}(B^{-1})
\]
\end{lem}
\begin{proof}
The proof is a simple calculation we omit. It is based on the
fact that $E_{ij}E_{k\l}=\d_{jk}E_{i\l}$.
\end{proof}
We now define direct sum of matrices by embedding the matrices in a larger matrix according to two index subsets.
\begin{defn}
For two subsets $S_{1},S_{2}\subset[N]$ and matrices $A\in \Mat_{|S_{1}|,|S_{2}|},B\in \Mat_{N-|S_{1}|,N-|S_{2}|}$,
we define their welding according to $S_1, S_2$ to be:
\[
A\oplus_{S_{1},S_{2}}B:=E_{S_{1},S_{2}}(A)+E_{S_{1}^{C},S_{2}^{C}}(B)
\]

and, for: $$h_{1}=(h_{1}^{1},h_{1}^{2})\in \Mat_{|S_{1}|}\times \Mat_{|S_{2}|}$$
 $$h_{2}=(h_{2}^{1},h_{2}^{2})\in \Mat_{N-|S_{1}|}\times \Mat_{N-|S_{2}|},$$
denote: \[h_{1}\oplus_{S_{1},S_{2}}h_{2}:=(h_{1}^{1}\oplus_{S_{1},S_{1}}h_{2}^{1},h_{1}^{2}\oplus_{S_{2},S_{2}}h_{2}^{2}).\] 
\end{defn}
\begin{rem}
Note that for any $S_{1},S_{2},S_{3}\subset[n]$ we have 
\begin{align*}
(A\oplus_{S_{1},S_{2}}B)(C\oplus_{S_{2},S_{3}}D) & =(AC\oplus_{S_{1},S_{3}}BD)\\
(h_{1}\oplus_{S_{1},S_{2}}h_{2})(k_{1}\oplus_{S_{1},S_{2}}k_{2}) & =(h_{1}^{1}k_{1}^{1}\oplus_{S_{1},S_{1}}h_{2}^{1}k_{2}^{1},h_{1}^{2}k_{1}^{2}\oplus_{S_{2},S_{2}}h_{2}^{2}k_{2}^{2})\\
 & =h_{1}k_{1}\oplus_{S_{1},S_{2}}h_{2}k_{2}
\end{align*}

for any $A,B,C,D$ matrices of appropriate dimensions and $h_{i},k_{i}$
tuples of matrices of dimensions as dictated by the definition of
$\oplus$.
\end{rem}
We first show that the action of welding together two elements of $\widetilde{H}_{n,k}$ is consistent
with the definitions of the characters $\psi$ and the involution $\tau$.
\begin{lem}
\label{lem:char-comp}Let $S\subset[n+m]$ and denote $S_{0}:=S\cap[n]$ and $\overline{S}=S\backslash[2n]-2n$. We assume $S_0=S\cap\{n+1,...,2n\}-n$.

If $h_{1}\in H_{|S_{0}|,|S_{0}|+|\overline{S}|}$, $h_{2}\in H_{n-|S_{0}|,m-|S_{0}|-|\overline{S}|}$
then $h:=h_{1}\oplus_{S,S}h_{2}$ is in $H_{n,m}$ and $\psi(h)=\psi(h_{1})\psi(h_{2})$.

It still holds if we replace $H$ with $H'$ (with the character $\psi\circ\tau$), and if we replace $\psi$ with $\psi_u$.
\end{lem}
\begin{proof}
We note that, writing for $i=1,2$:
\[
h_{i}=\left[\begin{array}{ccc}
g_{i} & u_{i} & a_{i}\\
 & g_{i} & b_{i}\\
 &  & k_{i}
\end{array}\right]
\]
then
\begin{align}
E_{S,S}(h_{1}) & =\left[\begin{array}{ccc}
E_{S_{0},S_{0}}(g_{1}) & E_{S_{0},S_{0}}(u_{1}) & E_{S_{0},\overline{S}}(a_{1})\\
 & E_{S_{0},S_{0}}(g_{1}) & E_{S_{0},\overline{S}}(b_{1})\\
 &  & E_{\overline{S},\overline{S}}(k_{1})
\end{array}\right].\label{eq:Eh}
\end{align}
and $E_{S^{C},S^{C}}(h_{2})$ has a similar form. Therefore $h$ has the appropriate block structure and it is easy to verify that $h$ satisfies the other requirements in the definition of $H$. Hence $h\in H_{n,m}$.
Moreover, 
\begin{align*}
\psi_u(h) & =\psi_u(E_{S,S}(h_{1})+E_{S^{C},S^{C}}(h_{2}))\\
 & =\psi_{0}\left(\tr((E_{S_{0},S_{0}}(u_{1})+E_{S_{0}^{C},S_{0}^{C}}(u_{2}))(E_{S_{0},S_{0}}(g_{1}^{-1})+E_{S_{0}^{C},S_{0}^{C}}(g_{2}^{-1})))\right)\\
 & =\psi_{0}(\tr(E_{S_{0},S_{0}}(u_{1}g_{1}^{-1})+E_{S_{0}^{C},S_{0}^{C}}(u_{2}g_{2}^{-1})))\\
 & =\psi(h_{1})\psi(h_{2}).
\end{align*}

Showing the property for $\psi$ is very similar, as is proving it for $H'$.
\end{proof}
\begin{lem}
\label{lem:inv-comp}Assume that $S_{2}\cap[2n]=2n+1-S_{1}\cap[2n]$
and $S_{1}\backslash[2n]=S_{2}\backslash[2n]$. Then $\tau(A\oplus_{S_{1},S_{2}}B)=\tau(A)\oplus_{S_{1},S_{2}}\tau(B)$. 
\end{lem}
\begin{proof}
We note that for example:
\begin{align*}
\tau(E_{S_{1},S_{2}}(A))= & \diag(w_{2n},I_{m-n})(E_{S_{1},S_{2}}(A))^{t}\diag(w_{2n},I_{m-n})\\
= & \diag(w_{2n},I_{m-n})(E_{S_{2},S_{1}}(A^{t}))\diag(w_{2n},I_{m-n})\\
  =&E_{S_{1},S_{2}}(\diag(w_{|S_{2}\cap[2n]|},I_{m-n})A^{t}\diag(w_{|S_{2}\cap[2n]|},I_{m-n})\\=&E_{S_{1},S_{2}}(\tau(A)).
\end{align*}
and use the linearity of $\tau$.
\end{proof}
We are now ready to relate the admissibility and $\psi$-$\tau$-invariance of $A$,$B$ and $A\oplus_{S_{1},S_{2}}B$.
\begin{lem}
\label{lem:ind-sym-G}Assume that $S_{1}$ and $S_{2}$ satisfy the
conditions in Lemmas \ref{lem:char-comp} and \ref{lem:inv-comp}. Then:

(i) If $A\oplus_{S_{1},S_{2}}B$ is admissible then
$A$ and $B$ are admissible.

(ii) $A$ and $B$ are both $\psi$-$\tau$-invariant
then $A\oplus_{S_{1},S_{2}}B$ is $\psi$-$\tau$-invariant.
\end{lem}
\begin{proof}
Denote $x:=A\oplus_{S_{1},S_{2}}B$, $\widetilde{\psi}(h_1,h_2)=\psi(\tau(h_{1})h_{2}^{-1})$ and the same for $\widetilde{\psi}_u$.

(i) Assume that $h_{1}\in \widetilde{H}_{A}$ and $h_{2}\in \widetilde{H}_{B}$. Using Lemmas
\ref{lem:char-comp} and \ref{lem:Eprops} we get that $h_{1}\oplus_{S_{1},S_{2}}(I,I), (I,I)\oplus_{S_{1},S_{2}}h_{2}\in \widetilde{H}_{x}$.
By the admissibility of $x$, $\widetilde{\psi}_u(h_{1}\oplus_{S_{1},S_{2}}(I,I))=\widetilde{\psi}_u(h_{1})$
and the same for $h_{2}$. Therefore, the $\psi_u$-admissibility of $x$ implies
the $\psi_u$-admissibility of $A$ and $B$. The proof for $\psi$-admissibility is identical.

(ii) Similarly, if $h_{1}.A=\tau(A)$ and $h_{2}.B=\tau(B)$ with
$\psi(h_{1})=\psi(h_{2})=1$ then $(h_{1}\oplus_{S_{1},S_{2}}h_{2}).x=\tau(x)$
and $\widetilde{\psi}(h_{1}\oplus_{S_{1},S_{2}}h_{2})=\widetilde{\psi}(h_{1})\widetilde{\psi}(h_{2})=1$.
Therefore, if $A$ and $B$ are both $\psi$-$\tau$-invariant then
$x$ is $\psi$-$\tau$-invariant.
\end{proof}
Practically, we are going to choose $A=I_{k}$ for some $k$ and use Lemma \ref{lem:ind-sym-G} to get that
we can assume $k=0$, or in other words, it is enough to show that if $B$ is admissible then it is $\psi$-$\tau$-invariant to get that if $x=E_{S_{1},S_{2}}(I_{k})+E_{S_{1}^{C},S_{2}^{C}}(B)$ is admissible then it is $\psi$-$\tau$-invariant,
under the conditions of the last Lemma.

Lastly, we show how for certain matrices the problem of showing that a coset in $G=\GL_{n+m}$ is admissible or $\psi$-$\tau$-invariant is equivalent to showing these properties for a matrix in $\Mat_{2n}$. 

Define an operator $C:\Mat_{n+m}\ra \Mat_{2n}$ which maps a matrix to its $2n\times 2n$ top left block, essentially cutting the matrix. We also extend the action of $\widetilde{H}$ on $\GL_{n+m}$ to an action on $\Mat_{n+m}$ and define admissibility and $\psi$-$\tau$-invariance in the same way as before.
\begin{lemma}\label{lem:redclas} Let $\frac n2\leq k\leq n$ and $t\leq \min \{m-n,k\}$.
The coset of the matrix of the form

\[
\eta_{A_0,B_0,A'_0,B'_0} = \left[\begin{array}{cc|cc|cc}
0_k &  &  & A_0 & & A'_0\\
 & I_{n-k} &  & \\
\hline  &  & I_{n-k} & \\
B_0 &  &  & 0_k\\
\hline  &  &  &  & I_{m-n-t}\\
 &  &  & B'_0 & &0_t 
\end{array}\right]\in \GL_{n+m}
\]
is admissible ($\psi$-$\tau$-invariant) in $\GL_{n+m}$ if and only if 
$C(\eta)$
is admissible ($\psi$-$\tau$-invariant) in $\Mat_{2n}$ with respect to the action of of $\widetilde{H}_{n,n}$.
\end{lemma}
\begin{proof}
We prove the claim regarding $\psi$-$\tau$-invariance. The only if direction is obvious. 

The strategy of this proof is to reduce the claim to representatives with a relatively simple form. We define:
\begin{defn}
We say that $\eta_{A,B,A',B'}$ is in reduced form if there exist $S_{1},S_{2}\subseteq[k]$ of size $k-t$ and a permutation
matrix $P$ such that $A=E_{S_{1},S_{2}}(P)$, $A'=E_{S_{1}^{C},[t]}(I_{t})$ and $B'=E_{[t],S_{2}^{t}}(I_{t})$.
\end{defn}

We claim that the coset of $\eta_{A_{0},B_{0},A'_{0},B'_{0}}$ has a representative in reduced form. 
To see that, note that for any matrices $$a=\left[\begin{array}{c|c}
* & 0_{n-k,2k-n}\\
\hline * & *
\end{array}\right],b=\left[\begin{array}{c|c}
* & 0_{2k-n,n-k}\\
\hline * & *
\end{array}\right]\in \GL_{k}$$ we can find $B_{1},A'_{1},B'_{1}$ such that $\eta_{A_{0},B_{0},A'_{0},B'_{0}}\sim\eta_{aA_{0}b,B_{1},A'_{1},B'_{1}}$. We choose $a,b$ such that $aA_{0}b=E_{S_{1},S_{2}}(P)$
for some $S_{1},S_{2}\subseteq[k]$
and a permutation matrix $P$. Acting with an appropriate element
in $\widetilde{H}$ we get an element in reduced form.  Note that $|S_{i}|=k-t$ follows now from the fact that the matrix is invertible.

Since $\eta_{1}\sim\eta_{2}$ implies $C(\eta_{1})\sim C(\eta_{2})$ and the $\psi$-$\tau$-invariance property doesn't depend on the representative, it is enough to show the claim for matrices in reduced form.

Assume that $\eta$ is in reduced form and $\tilde h.C(\eta)=\tau(C(\eta))$ with $\widetilde\psi(\tilde h)=1$. 
Then $C(\tau(\eta))=C((\tilde h \oplus_{[2n],[2n]}(I_{m-n},I_{m-n})).\eta)$.
Using this property, one can find $\tilde{h}'\in\widetilde{H}_{n,m}$ which
coincides with $\tilde h \oplus_{[2n],[2n]}(I_{m-n},I_{m-n})$ on the top left $2n\times2n$ block
and satisfies $\tau(\eta)=\tilde{h}.\eta$. This matrix shows that
$\eta$ is $\psi$-$\tau$-invariant.

The proof of the part regarding admissibility is very similar.
\end{proof}

\subsection{Verifying Proposition \ref{prop:adm-inv} on a dense subset}\label{sec:dense}
In this subsection we choose a set of representatives of the cosets $H'\backslash G/H$ and then divide the representatives into two types. We prove the geometric statement (Proposition \ref{prop:adm-inv}) for representatives of the first type. We will deal with the second type in later subsections.

First, recall that $H=H_0U$ and $H_0\isom \GL_n\times \GL_{m-n}$. Therefore, there exist characters $\psi_1,\psi_2$ of $F^\times$ such that 
 \[\psi(\diag(g,g,h))=\psi_1(\det(g))\psi_2(\det(h))\]
 for any $g\in \GL_n$ and $h\in \GL_{m-n}$. 
We denote: \[d(A,B)=\left[\begin{array}{ccc}
A & B & 0\\
 0& A&0\\
 0& 0 & I
\end{array}\right]\] and $\Delta(A):=d(A,0)$ and for $A,B\in \Mat_{n}$.

Write $\s_{k_1,k_2,t,s}$ for 
\[
\left[\begin{array}{ccc|ccc|ccc}
 &  &  &  & w_{k_1} & \\
 & I_{n-t-k_1} &  &  &  & \\
 &  &  &  &  &  &  &  & I_{t}\\
\hline  &  &  & I_{n-k_2-s} &  & \\
w_{k_2} &  &  &  &  & \\
 &  &  &  &  &  &  & I_{s}\\
\hline  &  &  &  &  &  & I_{m-n-s-t}\\
 &  &  &  &  & I_{k_2-k_1+s}\\
 &  & I_{k_1-k_2+t} &  &  & 
\end{array}\right]
\]
where 
\begin{align*}
(k_1,k_2,t,s)\in\Omega:=\left\{(k_1,k_2,t,s) \mid 
\begin{array}{l}
    0\leq k_1,k_2\leq n;\,0\leq s,t;\,k_2\leq t+k_1\leq n;\\k_1\leq s+k_2\leq n;
\,s+t\leq m-n
  \end{array}
\right\}.
\end{align*}

\begin{lemma}
\label{claim:reps}Denote:
\[
\gamma_{Y,Z,k_1,k_2,t,s}  :=\diag(Y,I_{m})\s_{k_1,k_2,t,s}\diag(I_{n},Z,I_{m-n}).
\]
Then the following is a complete set of representatives
of $H^{'}\backslash G/H$:
\begin{align*}
\{ \gamma_{Y,Z,k_1,k_2,t,s}\mid Y,Z\in \GL_{n},(k_1,k_2,t,s)\in\Omega\}
\end{align*}
\end{lemma}
\begin{proof}
Let $Q_1,Q_2$ be the standard parabolic subgroups corresponding to the partitions $\{n,n,m-n\}$ and $\{m-n,n,n\}$. Denote by $W_{n+m}$ and $W_{Q_i}$ the Weyl group of $G_{n+m}$ and the one corresponding to $Q_i$, for $i=1,2$. By the relative Bruhat decomposition, there exists a bijection $Q_1\backslash G/Q_2\longleftrightarrow W_{Q_1}\backslash W_{n+m}/W_{Q_2}$.
The latter has the following complete set of representatives:
\begin{align*}
\{\left[\begin{array}{ccc|ccc|ccc}
 &  &  &  &  &  & I_{m-n-s-t}\\
 &  &  &  &  & I_{s+k_2-k_1}\\
 &  & I_{t+k_1-k_2} &  &  & \\
\hline   &  &  &  & w_{k_1} & \\
 & I_{n-t-k_1} &  &  &  & \\
 &  &  &  &  &  &  &  & I_{t}\\
\hline &  &  & I_{n-k_2-s} &  & \\
w_{k_2} &  &  &  &  & \\
 &  &  &  &  &  &  & I_{s}\\
\end{array}\right] \}
\end{align*}
with $\mid (k_1,k_2,t,s)\in\Omega$.
Using $Q_1=\{\diag(I_{n},Z,I_{m-n}) \mid Z\in \GL_{n}\}\cdot H$
and 
\[
\left(\begin{array}{ccc}
 &  & I_{n}\\
I_{n}\\
 & I_{m-n}
\end{array}\right)H'\cdot\{\diag(Y,I_{n},I_{m-n}) \mid Y\in \GL_{n}\}\left(\begin{array}{ccc}
 & I_{n}\\
 &  & I_{m-n}\\
I_{n}
\end{array}\right)=Q_2
\]
 we get the result.
\end{proof}
Fix a representative $\gamma:=\gamma_{Y,Z,k_1,k_2,t,s}$. Write $$Y=\left(\begin{array}{ccc}
Y_{1} & Y_{2} & Y_3\\
Y_{4} & Y_{5} & Y_6\\
Y_{7} & Y_{8} & Y_9
\end{array}\right), Z=\left(\begin{array}{ccc}
Z_{1} & Z_{2} & Z_3\\
Z_{4} & Z_{5} & Z_6\\
Z_{7} & Z_{8} & Z_9
\end{array}\right)$$ with $Y_{4}\in \Mat_{k_2},Z_{4}\in \Mat_{k_1, k_2},Z_{9}\in \Mat_{s+k_2-k_1,t+k_1-k_2}$ and $Y_{9}\in \Mat_{s, t+k_1-k_2}$.
The representative has the form:

\[
\gamma=\left[\begin{array}{ccc|ccc|ccc}
 & Y_{2} &  & C_1 & C_2 & C_3 &  &  & Y_{3}\\
 & Y_{5} &  & C_4 & C_5 & C_6 &  &  & Y_{6}\\
 & Y_{8} &  & C_7 & C_8 & C_9 &  &  & Y_{9}\\
\hline  &  &  & Z_{1} & Z_{2} & Z_{3}\\
w_{k_2} &  &  &  &  & \\
 &  &  &  &  &  &  & I_{s}\\
\hline  &  &  &  &  &  & I_{m-n-s-t}\\
 &  &  & Z_{7} & Z_{8} & Z_{9}\\
 &  & I_{t+k_1-k_2} &  &  & 
\end{array}\right]
\]
with some blocks $C_i$ of the appropriate dimensions.

We need two lemmas to rule out some non-admissible cosets.
\begin{lem}
\label{lem:col_t}
Assume that $[\gamma]$ is admissible.
\begin{enumerate}[(i)]
    \item Let $1\leq i\leq n+m$. If the first $n$ elements in $[\gamma]_{i,\ra}$ are zero then the elements in places $2n+1-t-k_1+k_2,...,2n$ in the row are zero as well.
\item
Similarly, Let $1\leq j\leq n+m$. If the column $[\gamma]_{\da,j}$ is zero in the indices $n+1,...,2n$ then it is zero in the indices $n-s+1,...,n$ as well.
\end{enumerate}
\end{lem}
\begin{proof}
\begin{enumerate}[(i)]
    \item
For any $1\leq j\leq k_1-k_2+t$, write $A=I_{n+m}+E_{m+n-t-k_1+k_2+j,i}$ and $B=I_{n+m}-\sumud{\l=n+1}{n+m}\gamma_{i,\l}E_{n-t-k_1+k_2+j,\l}$.
Observe that $A\in H'$, $B\in H$ and $A\gamma B=\gamma$. From admissibility
we get that $\psi_u(A)=1=\psi_u(B)^{-1}=(\psi_{0}(-\gamma_{i,2n-t-k_1+k_2+j}))^{-1}$,
i.e. $\gamma_{i,2n-t-k_1+k_2+j}=0$.
\item Let $1\leq i\leq s$.
Write $B=I_{n+m}+E_{j,n+m-s-t+i}$ and 
\[
A=I_{n+m}-\sum_{\mathclap{\l\in[n]\cup 2n+[m-n]}}\gamma_{\l,j}E_{\l,2n-s+i}
.\]
Again $A\in H'$, $B\in H$ and $A\gamma B=\gamma$. Therefore $\psi_u(A)=1=\psi_u(B)^{-1}=(\psi_{0}(\gamma_{n-s+i,j}))^{-1}$
and $\gamma_{n-s+i,j}=0$.\qedhere
\end{enumerate}

\end{proof}
We have an immidiate corollary from the last lemma.
\begin{cor}\label{cor:van}
If  $[\gamma]$ is admissible then:
\begin{enumerate}[(i)]
    \item $Z_3$, $Z_9$, $Y_8$ and $Y_9$ are zero (unconditionally).
    \item $Y_2=0$ implies $C_3=0$.
    \item $Y_8=0$ implies $C_9=0$.
    \item $Z_2=0$ implies $C_8=0$.
\end{enumerate}
\end{cor}
We rule out more cosets using the following lemma. This Lemma is adapted from \cite[Lemma 3.6]{Nien}. 
\begin{lemma}
\label{claim:Y2=Z2}Assume that $[\gamma]$ is admissible. Then
$Y_{2}=Z_{2}$. 
\end{lemma}
\begin{proof}
Let $$\tilde{X}:=\left(\begin{array}{ccc}
 & 0_{k_1,k_2}\\
X\\
 &  & 0_{t, s}
\end{array}\right)\in \Mat_{n}$$ for any $X\in \Mat_{n-t-k_1,n-k_2-s}$. We have $
d(I_{n}, Y\tilde{X})\gamma d(I_{n},-\tilde{X}Z)  =\gamma
$
and by admissibility $$\tr(Y\tilde{X})=\tr(Y_2{X})=\tr(\tilde{X}Z)=\tr({X}Z_2)$$
for all $X\in \Mat_{n-t-k_1,n-k_2-s}$. 
\end{proof}
We will finish the case of $Y_2=0$ in a later subsection. For now, we state the following proposition.

\begin{prop}\label{prop:Y2}
Assume $n-t-k_1,n-k_2-s>0$. If the coset of $\gamma$ is admissible and $Y_2=Z_2=0$, then it is $\psi$-$\tau$-invariant.
\end{prop}
We are now ready to prove Proposition \ref{prop:adm-inv} assuming Proposition \ref{prop:Y2}.
\begin{proof}[Proof of Proposition \ref{prop:adm-inv} assuming Proposition \ref{prop:Y2}]
We prove it by induction on $n$ (recall
that we assume $n\leq m$). 

\textbf{Induction base}: The case of $n=0$ is easy since $H^{0,m}=H'^{0,m}=\GL_{m}(F)$,
hence there is only one coset.

\textbf{Induction step}: Assume that $n\geq1$. Using the notation
from above, 
take a representative
$\gamma:=\gamma_{Y,Z,k_1,k_2,t,s}$. 

We first deal with the case of $n-t-k_1=0$. Use Lemma \ref{lem:col_t} to deduce
that columns $2n-t-k_1+k_2+1,...,2n$ must be zero. That contradicts $\gamma\in G$ unless $k_1+t-k_2=0$, which implies $k_2=n$ and $s=0$. By Lemma \ref{lem:redclas} in this case it is enough to show that the coset of the top left $2n\times 2n$ sub-matrix is $\psi$-$\tau$-invariant. This restriction has the form
\[
C(\gamma)=\left[\begin{array}{cc}
 & A\\
w_{n}
\end{array}\right].
\]
Using \cite{sim}, we choose a matrix $X\in \GL_n$ such that $Xw_nAX^{-1}=A^tw_n$ we get that $$\Delta(w_nXw_n)C(\gamma)\Delta(X^{-1})=C(\gamma),$$ showing that $C(\gamma)$ is $\psi$-$\tau$-invariant. Therefore by Lemma \ref{lem:redclas}, $\gamma$ is $\psi$-$\tau$-invariant well, solving the case of $n-t-k_1=0$.
Similarly, $n-k_2-s=0$ implies $n=k_1$ and is solved in the same fashion.

We return to the general case and assume $n-t-k_1,n-k_2-s>0$. We have solved the case of $Y_{2}\neq Z_{2}$ in Lemma \ref{claim:Y2=Z2} and postpone the treatment of the case of $Y_{2}=Z_{2}=0$ (see Proposition \ref{prop:Y2}). It remains to solve the case of $Y_{2}=Z_{2}$ and $\text{rank} Y_{2}>0$. By Corollary \ref{cor:van},
 $Y_{8}=0$ and $Z_{3}=0$. Choose matrices $A\in \GL_{n-k_2-s},B\in \GL_{n-t-k_1}$
such that \[AY_{2}B=\left[\begin{array}{cc}
0_{1, n-t-k_1-1} & 1\\
* & 0_{n-k_2-s-1,1}
\end{array}\right]=:Y_{2}'.\] Denote $\gamma'=\Delta(A,I_{k+s})\gamma\Delta(I_{k},B,I_{t})$,
$Y_{5}'=Y_{5}B$ and $Z_{1}'=AZ_{1}$. Choose matrices $C\in \Mat_{k,n-k-s},D\in \Mat_{n-k-t,k}$
such that last column of $Y_{5}'-CY_{2}'$ and the first row
of $Z_{1}-Y_{2}'D$ are zero. Multiplying $\gamma'$ 
by $\Delta(\left[\begin{array}{ccc}
I_{n-k_2-s}\\
-C & I_{k_2}\\
 &  & I_{s}
\end{array}\right])$ on the left and by $\Delta(\left[\begin{array}{ccc}
I_{k_2}\\
-D & I_{n-t-k_1}\\
 &  & I_{t+k_1-k_2}
\end{array}\right])$ on the right we get:

\[
\gamma\sim\left[\begin{array}{cccc|cccc|ccc}
 &  & 1 &  & * & * & * & * &  &  & *\\
 & * &  &  & * & * & * & * &  &  & *\\
 & * &  &  & * & * & * & * &  &  & *\\
 & * &  &  & * & * & * & * &  &  & *\\
\hline  &  &  &  &  &  & 1 & \\
 &  &  &  & * & * &  & *\\
w_{k_2} &  &  &  &  &  &  & \\
 &  &  &  &  &  &  &  &  & I_{s}\\
\hline  &  &  &  &  &  &  &  & I_{m-n-s-t}\\
 &  &  &  & * & * & * & *\\
 &  &  & I_{t-k_1-k_2} &  &  &  & 
\end{array}\right].
\]

Therefore:
\[
\gamma\sim\left[\begin{array}{ccc|ccc|ccc}
 &  & 1 &  &  & \\
M_{1} &  &  & M_{2} &   M_{3}& &  &  & M_{4}\\
\hline  &  &  &  &  &1 \\
M_{5} &  &  & M_{6} &   0&\\
 &  &  &  &  &  &  & I_{s}\\
\hline  &  &  &  &  &  & I_{m-n-s-t}\\
 &  &  & M_{7} &   M_{8}&\\
 & I_{t-k_1-k_2}  & &  &  & 
\end{array}\right].
\]
for some matrices $M_i$ of appropriate dimensions. Define:
\[
\hat{\gamma}:=\left[\begin{array}{cc|cc|ccc}
M_{1} &  & M_{2} & M_{3} &  &  & M_{4}\\
\hline M_{5} &  & M_{6} & 0\\
 &  &  &  &  & I_{s}\\
\hline  &  &  &  & I_{m-n-s-t}\\
 &  & M_{7} & M_{8}\\
 & I_{t} &  & 
\end{array}\right]\in G_{n+m-2}.
\]
By Lemma \ref{lem:ind-sym-G}
, $\hat \gamma$ is admissible. The induction hypothesis implies that $\hat \gamma$ is $\psi$-$\tau$-invariant, and using Lemma \ref{lem:ind-sym-G} again we deduce that $\gamma$ is $\psi$-$\tau$-invariant.\end{proof}
\subsection{An alternative geometric statement}\label{sec:genset}
We now describe a problem similar to Proposition \ref{prop:adm-inv}. In Subsection \ref{sec:red} we show by induction that it implies Proposition \ref{prop:Y2}. In Subsection \ref{sec:simprob} we prove the similar problem itself also by induction.

Let $n\in\bb N$ and choose $k\in\bb N$ with $2k\leq n$. For a matrix $A\in \Mat_{n}$ we denote:
\[
A=\left[\begin{array}{ccc}
A_{1} & A_{2} & A_{3}\\
A_{4} & A_{5} & A_{6}\\
A_{7} & A_{8} & A_{9}
\end{array}\right]
\]
with $A_{3},A_{5}\in \Mat_{k}$. Let $X:=X_{n,k}=\{x\in \Mat_{n} \mid x_{3}=0\}$.
Let $P_{1},P_{2}$ be the block lower triangular subgroups of $\GL_{n}$
corresponding to the partitions $\{k,k,n-2k\}$ and $\{n-2k,k,k\}$
respectively. We enumerate the $9$ blocks of elements in each of them
similarly. Define 
\[
T:=T_{n,k}=\{(a,b)\in P_{1}\times P_{2} \mid a_{1}=a_{5},b_{5}=b_{9},a_{9}=b_{1}\}.
\]
 We define an action of $T$ on $X_{n,k}$: $(a,b).x=axb^{-1}$ and choose an involution $\s(x)=w_{n}x^{t}w_{n}$.  

We choose below a generating set $T=<E_{i}>$. Using it, we denote $F_{n,k}:=FG(\{E_{i}\})$.
We define a map $\nu_{x}:F_{n,k}\ra\widetilde{H}_{n+k,n+k}$ which will
translate a sequence of actions of $T$ on $X$ to an action of $\widetilde{H}$
on $G$. For an $x\in X$ we define functions $f_{x},f_x^u:F_{n,k}\ra F$
by $f_{x}(w)=\psi(\nu_{x}(w))$ and $f^u_{x}(w)=\psi_u(\nu_{x}(w))$. These functions send an action on
$X$ to the value of the characters $\psi$ and $\psi_u$ on the corresponding action on $G$.

\begin{defn}
We say that an $x\in X$ is $*$-admissible if $x_{2}=x_{6}$ and for
any $w\in F_{n,k}$ satisfying $w.x=x$ we have $f_{x}^u(w)=1$. We
say that $x$ is $*$-invariant if there exist $w\in F_{n,k}$ such
that $\sigma(x)=w.x$ and $f_{x}(w)=1$.
\end{defn}
Using these definitions, we state an new geometric statement.
\begin{lem}
\label{lem:Alternative-geometric-statement}(Alternative geometric
statement) If $x$ is $*$-admissible then it is $*$-invariant.
\end{lem}

In the next subsection we show that Lemma \ref{lem:Alternative-geometric-statement}
implies Proposition \ref{prop:Y2} and Subsection \ref{sec:simprob} it we prove Lemma
\ref{lem:Alternative-geometric-statement}. 
In the reminder of this subsection, we define the map $\nu_x$ explicitly and prove a few properties of $\nu_x$, $f_x$ and $f_x^u$.

We define for $X,Z\in \GL_{k}$ , $Y\in \GL_{n-2k}$, $a\in \Mat_{k},b,c^{t}\in \Mat_{n-2k,k}$:
\begin{align*}
E_{1}(X,Y,Z) & =(\diag(X,X,Y),\diag(Y,Z,Z))\\
E_{2}(a,b) & =\left[\begin{array}{ccc}
I\\
a & I\\
b &  & I
\end{array}\right],E_{3}(b)=\left[\begin{array}{ccc}
I\\
 & I\\
 & b & I
\end{array}\right]
\end{align*}
and identify $E_{i}$ with $(E_{i},I_{n})\in T$ for $i=2,3$. Similarly,
we define $E_{4}(c)=\s(E_{3}(\s(c)))$ and $E_{5}(a,c)=\s(E_{2}(\s(a),\s(c)))$
and embed them in the second coordinate of $T$. Clearly,
\begin{align*}
T= & <E_{1}(X,Y,Z),E_{2}(a,b),E{}_{3}(b),E_{4}(c),E_{5}(a,c) \mid \\
 & X,Z\in \GL_{k},Y\in \GL_{n-2k},a\in \Mat_{k},b,c^{t}\in \Mat_{n-2k,k}>.
\end{align*}

We define a map $\nu_{A}:F_{n,k}\ra\widetilde{H}_{n+k,n+k}$ by setting:
\[
\nu_{A}(E_{1}(X,Y,Z))=(\Delta(\diag(X,X,Y,Z)),\Delta(\diag(X,Y,Z,Z)))
\]
\[
\nu_{A}(E_{2}(a,b))=\left(\Delta\left(\begin{array}{cccc}
I\\
a & I\\
b &  & I\\
 &  &  & I
\end{array}\right),d\left(I_{n},\left[\begin{array}{cc}
a\\
b\\
 & 0_{k,n-k}
\end{array}\right]\right)\right)
\]
and finally, writing:
\[
\alpha_{A,b}:=\left[\begin{array}{cc}
A_{1}b\\
A_{4}b\\
A_{7}b+bA_{4}b\\
 & 0_{k,n-k}
\end{array}\right]
,\]
we define:
\[
\nu_{A}(E_{3}(b))=\left(d\left(\left[\begin{array}{cccc}
I\\
 & I\\
 & b & I\\
 &  &  & I
\end{array}\right],\alpha_{A,b}\right),\Delta\left(\begin{array}{ccc}
I\\
b & I\\
 &  & I_{2k}\\

\end{array}\right)\right).
\]
To explain why we have defined the map in this way, we consider the map $$A\mt \eta_A:=\left[\begin{array}{cc|cc}
 &  &  & A\\
 & I_{k}\\
\hline  &  & I_{k}\\
I_{n} & 
\end{array}\right].$$
We note that in each case $\eta_{E_{i}A}=\nu_{A}(E_{i})\eta_{A}$.
One can extend the definition to $E_{4}$ and $E_{5}$ by applying
$\s$ on both sides of $\eta_{E_{i}A}=\nu_{A}(E_{i})\eta_{A}$ for
$i=2,3$ to get elements of $\widetilde{H}$ which satisfy this identity
for $i=4,5$. We extend the definition to $F_{n,k}$ by $\nu_{x}(w_{1}w_{2})=\nu_{w_{2}x}(w_{1})\nu_{x}(w_{2})$ and deduce that:
\begin{lemma}
\label{lem:nueta}
 For any $w\in F_{n,k}$ holds $\eta_{w.A}=\nu_{A}(w)\eta_{A}$.
\end{lemma}
We will use this property in the next subsection to relate $*$-admissibility and $*$-invariance of $A$ to admissibility and $\psi$-$\tau$-invariance of $\eta_A$.

We note that:
\begin{align*}
f_{x}: & E_{1}(X,Y,Z)\mt\psi_{1}(\frac{\det(X)}{\det(Z)}),E_{2}(a,b)\mt\psi_{0}(\tr(-a)),\\
 & E_{3}(b)\mt\psi_{0}(\tr(x_{1}b)),E_{4}(c)\mt\psi_{0}(\tr(-cx_{9})),E_{5}(a,c)\mt\psi_{0}(\tr(a))
\end{align*}
and that $f_{x}(w_{1}w_{2})=f_{w_{2}x}(w_{1})f_{x}(w_{2})$ for any
$w_{i}\in F_{n,k}$. The same holds for $f_x^u$ apart from the value on $E_{1}(X,Y,Z)$, which is always $1$.

We remark that the reason we act on $X_{n,k}$ with a free group instead of acting directly with $T_{n,k}$ is that the group $T_{n,k}$ has more relations than the corresponding actions in $\widetilde H$ preventing us from translating actions of $T_{n,k}$ to actions of $\widetilde H$ directly.

We also need a version of Lemmas \ref{lem:char-comp}-- \ref{lem:ind-sym-G} for this problem.
\begin{lem}
\label{lem:char-comp-T}Assume that $S_{1},S_{2}\subset[n]$ satisfy the following conditions:
\begin{itemize}
    \item 

$S_{1}\cap[k]=S_{0}$ and $S_{1}\cap\{k+1,...,2k\}=S_{0}+k$ for
some $S_{0}\subset[k]$.

\item $n+1-S_{2}\backslash[n-2k]=S_{1}\cap[2k]$.

\item $S_{1}\backslash[2k]-2k=S_{2}\cap[n-2k]:=\overline{S}$.
\end{itemize}
Define a homomorphism $\lambda:F_{2|S_{0}|+|\overline{S}|,|S_{0}|}\times F_{n-2|S_{0}|-|\overline{S}|,k-|S_{0}|}\ra F_{n,k}$
by defining 
\begin{align*}
(E_{i},I_{n-2|S_{0}|-|\overline{S}|}) & \mt E_{i}\oplus_{S_{1},S_{2}}I_{n-2|S_{0}|-|\overline{S}|}\\
(I_{2|S_{0}|+|\overline{S}|},E_{i}) & \mt I_{2|S_{0}|+|\overline{S}|}\oplus_{S_{1},S_{2}}E_{i}
\end{align*}
and extending in the usual way. Then $f_{A\oplus_{S_{1},S_{2}}B}(\lambda(w_{1},w_{2}))=f_{A}(w_{1})f_{B}(w_{2})$
for all $w_{1}\in F_{2|S_{0}|+|\overline{S}|,|S_{0}|}$, $w_{2}\in F_{n-2|S_{0}|-|\overline{S}|,k-|S_{0}|}$ and matrices $A, B$ of appropriate dimensions. The same holds for $f^u$.
\end{lem}
\begin{lem}
\label{lem:inv-comp-T} Assume that $p\in \GL_{n-2k}$
is a permutation matrix which preserves the order of $\overline{S}$
and $\overline{S}^{C}:=[n-2k]\backslash \overline{S}$ (explicitly, if $i<j\in\overline{S}$ or $i<j\in\overline{S}^{C}$
then $p(i)<p(j)$) and denote $p_{1}:=I_{2k}\oplus_{[2k],[2k]}p$ and $  p_{2}:=p\oplus_{[n-2k],[n-2k]}I_{2k}$. Assume also that $n+1-S_{2}=p_{1}S_{1}$
and $n+1-S_{1}=p_{2}S_{2}$. Then $\s(A\oplus_{S_{1},S_{2}}B)=E_{1}(I_{k},p,I_{k})(\s(A)\oplus_{S_{1},S_{2}}\s(B))$.
\end{lem}
\begin{lem}
\label{lem:ind-sym-T}Assume that $S_{1}$ and $S_{2}$ satisfy the
conditions in Lemmas \ref{lem:char-comp-T} and \ref{lem:inv-comp-T} for some permutation $p\in \GL_{n-2k}$.
If $A\oplus_{S_{1},S_{2}}B$ is $*$-admissible then
$A$ and $B$ are $*$-admissible and if
$A$ and $B$ are both $*$-invariant
then $A\oplus_{S_{1},S_{2}}B$ is $*$-invariant.
\end{lem}

The proofs of these lemmas are very similar to the proofs of Lemmas \ref{lem:char-comp}--\ref{lem:ind-sym-G}.
\subsection{Reduction of Proposition \ref{prop:Y2} to Lemma \ref{lem:Alternative-geometric-statement}}\label{sec:red}
We now explain how the alternative geometric statement implies the claim for the representatives we have not solved yet. Recall that in Lemma \ref{claim:reps} and the discussion following it we defined representatives $\gamma=\gamma_{Y,Z,k_1,k_2,t,s}$ and divided $Y$ and $Z$ to blocks $Y_i,Z_i$ with $1\leq i\leq 9$. Using this notation, we remind the reader that the cosets we haven't solved yet are those with $Y_2=Z_2=0$.

\begin{lemma}
\label{claim:Y2=0}Assume that $[\gamma]$ is admissible and
$Y_{2}=Z_{2}=0$. Then $s=t+k_1-k_2=0$. 
\end{lemma}
\begin{proof}
By Corollary \ref{cor:van}, $Y_{8}=0$. Since $\gamma$ is invertible,
the columns of $Y_{5}$ must be linearly independent. We deduce that
$k_2\geq n-t-k_1$. Choose a matrix $A\in \GL_{k_2}$ s.t.
$AY_{5}=\left[\begin{array}{c}
0\\
I_{n-t-k_1}
\end{array}\right]$. Multiplying $\gamma$ by $\Delta(I_{n-k_2-s},A,I_{s})$ on the left we get that:
\[
\gamma\sim\left[\begin{array}{ccc|ccc|ccc}
 & 0 &  & * & * & * &  &  & *\\
 & I_{n-t-k_1} &  & * & * & * &  &  & *\\
 & 0_{s,n-t-k_1} &  & * & * & * &  &  & *\\
\hline  &  &  & Z_{1} & 0 & Z_{3}\\
Aw_{k_2} &  &  &  &  & \\
 &  &  &  &  &  &  & I_{s}\\
\hline  &  &  &  &  &  & I_{m-n-s-t}\\
 &  &  & Z_{7} & Z_{8} & Z_{9}\\
 &  & I_{t+k_1-k_2} &  &  & 
\end{array}\right]
\]
By Corollary \ref{cor:van}, $Z_3=0$. Similarly we deduce $k_2\geq n-k_2-s$ from the linear independence of the rows of $Z_1$ and choose $B\in \GL_{k_2}$ such that $Z_{1}B=\left[\begin{array}{cc}
I_{n-k_2-s} & 0\end{array}\right]$.
Applying Lemma \ref{lem:col_t} and acting with an element of $\widetilde{H}$, we get
\[
\gamma\sim\left[\begin{array}{ccc|ccc|ccc}
 & 0 &  & 0 & * & 0 &  &  & *\\
 & I_{n-t-k_1} &  & 0 & 0 & 0 &  &  & 0\\
 & 0 &  & 0 & 0 & 0_{s,t+k_1-k_2} &  &  & 0\\
\hline  &  &  & I_{n-k_2-s} & 0 & 0\\
Aw_{k_2}B &  &  &  &  & \\
 &  &  &  &  &  &  & I_{s}\\
\hline  &  &  &  &  &  & I_{m-n-s-t}\\
 &  &  & 0 & * & 0\\
 &  & I_{t+k_1-k_2} &  &  & 
\end{array}\right].
\]
Since we are in $\GL_{n+m}$, we must have $t+k_1-k_2=s=0$. 
\end{proof}
Denote $k:=k_2=t+k_1$ and $e:=2k-n\geq 0$. 
Note that in the last lemma we have got that under the conditions of Proposition \ref{prop:Y2} we can assume that $\gamma$ has the form
\[
\gamma = \left[\begin{array}{cc|cc|cc}
 &  &  & A & & A'\\
 & I_{n-k} &  & \\
\hline  &  & I_{n-k} & \\
B &  &  & \\
\hline  &  &  &  & I_{m-n-t}\\
 &  &  & B' & & 
\end{array}\right]
\]
with some $B\in \GL_k$, $A\in \Mat_k$ and $A',B'^t\in \Mat_{k,t}$.

By Lemma \ref{lem:redclas}, it is enough to prove the claim for cosets in $\Mat_{2n}$ which contain a representative of the form
\[
\eta_{A,B}=\left[\begin{array}{cc|cc}
 &  &  & A\\
 & I_{n-k}\\
\hline  &  & I_{n-k}\\
B & 
\end{array}\right]\in \Mat_{2n}
\]
with $B\in \GL_k$. 

An easy computation shows that if
\[
B'=\left[\begin{array}{cc}
X_1 & X_2\\
 & X_3
\end{array}\right]B\left[\begin{array}{cc}
X'_1 & X'_2\\
 & X'_3
\end{array}\right].
\]
with $X_1,X_3'\in \GL_{n-k}$ and $X_1,X_3'\in \GL_{e}$
then there exist $A'\in \Mat_{k}$ such that $\eta_{A,B}\sim\eta_{A',B'}$. Therefore we can assume \begin{equation*}\label{eq:B}
B=\left[\begin{array}{cc|cc}
0 & I_{\l} & 0 & 0\\
0 & 0 & I_{e-\l} & 0\\
\hline 0 & 0 & 0 & I_{\l}\\
I_{r} & 0 & 0 & 0
\end{array}\right]
\end{equation*}
for some $r\leq n-k$ and $\l:=n-k-r\leq e$.

Write:
\[
A=\left[\begin{array}{cc}
 A_{1} & A_{2}\\
 A_{3} & A_{4}\\
 \end{array}\right],\,
 B^{-1}=\left[\begin{array}{cc}
 B'_{1} & B'_{2}\\
 B'_{3} & B'_{4}\\
 \end{array}\right]
\]
with $A_{2},B'_{2}\in \Mat_{n-k}$. Then
\[
\Delta(\left[\begin{array}{ccc}
I_{n-k} &  & \\
 & I_{e} & \\
X &  & I_{n-k}
\end{array}\right]
)\cdot\eta_{A,B}\cdot 
d(I_n,
\left[\begin{array}{ccc}
-B'_2X &  & \\
-B'_4X &  & \\
 & -XA_{1} & -XA_{2}
\end{array}\right]
)=\eta_{A,B}. 
\]
Therefore if $[\eta_{A,B}]$ is admissible, $\tr(-B'_2X-XA_{2})=0$ for all $X\in \Mat_{n-k}$. We deduce that $A_2=-B'_2=\left[\begin{array}{cc}
0 & -I_{r}\\
0 & 0 
\end{array}\right]$. We get that $\eta_{A,B}$ is equivalent to a matrix of the form:
\[
\left[\begin{array}{ccc|ccc}
 &  &  &  &  & I_{r}\\
 &  &  &  & *\\
 &  & I_{n-k}\\
\hline  &  &  & I_{n-k}\\
 & * & \\
-I_{r} &  & 
\end{array}\right]
\]
on which we use Lemma \ref{lem:ind-sym-G} with $S_1=S_2=\{[r],n-r+[2r],2n+1-[r]\}$ to get that we can assume $r=0$. We have reduced the problem to representatives of the form:
\[
\eta_{A}:=\left[\begin{array}{cc|cc}
 &  &  & A\\
 & I_{\l}\\
\hline  &  & I_{\l}\\
I_k & 
\end{array}\right]
\]
with $A\in X^{k,\l}$ (note that $\l=n-k\leq e$). Write $$A=\left[\begin{array}{ccc}
A_{1} & A_{2} & 0_{\l}\\
 A_{4} & A_{5} & A_{6}\\
 A_{7} & A_{8} & A_{9}
\end{array}\right].$$
with $A_5\in \Mat_\l$.
Let $X\in \Mat_{\l}$. We denote:
$$M_1=d(\left[\begin{array}{cccc}
I_{\l}\\
 & I_{\l}\\
 &  & I_{k-2\l}\\
 & -X &  & I_{\l}
\end{array}\right],\left[\begin{array}{cc}
-A_{2}X\\
-A_{5}X\\
-A_{8}X\\
& 0_{\l,k}\\
\end{array}\right])$$
$$M_2=d(\left[\begin{array}{cccc}
I_{\l}\\
 & I_{k-2\l}\\
X &  & I_{\l}\\
 &  &  & I_{\l}
\end{array}\right],\left[\begin{array}{cccc}
\\
\\
\\
XA_{5}X & XA_{4} & XA_{5} & XA_{6}
\end{array}\right])$$
and observe that 
$
M_1\eta_{A}M_2=\eta_{A}
$
for any $X\in \Mat_{\l}$. As before, we deduce $A_{2}=A_{6}$.

Assume that $[\eta_{A}]$ is admissible. We want to show that $\eta_A$ is $\psi$-$\tau$-invariant. Assume
that $w.A=A$ for some $w\in F_{n,k}$ with $f^u_A(w)=1$. Then, using Lemma \ref{lem:nueta}, $\eta_{w.A}=\nu_{A}(w)\eta_{A}=\eta_{A}$
and by the admissibility of $[\eta_{A}]$ we deduce that $\psi_u(\nu_{A}(w))=f_{A}^u(w)=1$. Together with $A_2=A_6$ we deduce that $A$ is $*$-admissible. 
Using Lemma \ref{lem:Alternative-geometric-statement} which we prove in the next Subsection we get that
there exist $w'\in F_{n,k}$ such that $\sigma(A)=w'.A$ and $f_{A}(w')=1$.
Therefore $\tau(\eta_{A})=\eta_{\s(A)}=\eta_{w'.A}=\nu_{A}(w')\eta_{A}$
with $\psi(\nu_{A}(w'))=f_{A}(w')=1$.
This concludes the case of $Y_2=0$. 
\subsection{Proof of Lemma \ref{lem:Alternative-geometric-statement}}\label{sec:simprob} 
We prove the alternative geometric statement by induction on $n$. If $k=0$ then $A$ degenerates
to the block $A_{7}$. We choose $Y\in \GL_{n}$ such that $YAY^{-1}=w_{n}A^{t}w_{n}$ as we did in the proof of Proposition  \ref{prop:adm-inv}.
Since $(Y,Y)\in T_{n,0}$ has a trivial associated character, $A$
is $*$-invariant.

Assume that $k>0$ and $A=\left[\begin{array}{c|c|c}
A_{1} & A_{2} & 0\\
\hline A_{4} & A_{5} & A_{6}\\
\hline A_{7} & A_{8} & A_{9}
\end{array}\right]$ is $*$-admissible. Then $A_{2}=A_{6}$. Find $X,Z$ such that $XA_{2}Z^{-1}=\left[\begin{array}{cc}
0 & I_{d}\\
0 & 0
\end{array}\right]$ for some $d\in\bb N$. We act on $A$ with $E_{1}(X,I_{n-2k},Z)$
and then with suitable $E_{2},E_{5}$ to get an element of the form:
\[
A':=\left[\begin{array}{c|cc|cc}
0 & 0 & I_{d} & 0 & 0\\
A'_{1} & 0 & 0 & 0 & 0\\
\hline 0 & 0 & 0 & 0 & I_{d}\\
A'_{4} & A'_{5} & 0 & 0 & 0\\
\hline A'_{7} & A'_{8} & 0 & A'_{9} & 0
\end{array}\right]
\]

We now apply Lemma \ref{lem:ind-sym-T} on $S_{1}=[d]\cup k+[d]$, $S_{2}=n+1-S_{1}$ and trivial permutation $p$ to deduce that we can assume $d=0$. For any matrix $X\in \Mat_{k}$
such that $X'_1A=0$ we have $E_{2}(X,0)A=A$, hence $\tr(X)=0$. Consequently the rows
of $A'_{1}$ are linearly independent. In particular $k\leq n-2k$.
Similarly, the columns of $A'_{9}$ are linearly independent as well. Acting
with an appropriate element of $T$ we get:
\[
A'\sim\left[\begin{array}{cc|c|c}
I_{k} &  & \\
\hline  & B_{4} & B_{5}\\
\hline  & B_{7}^{1} & B_{8}^{1} & B_{9}^{1}\\
 & B_{7}^{2} & B_{8}^{2} & B_{9}^{2}
\end{array}\right]:=B
\]

with $B_{9}^{1}\in \Mat_{k}$. For any $X\in \Mat_{k}$, acting with
$E_{2}(B_{5}X,0)E_{4}(\left[\begin{array}{cc}
X & 0\end{array}\right])$ stabilises $B$ and has a character $\psi_{0}(\tr(-B_{5}X-X B_{9}^{1}))$.
Therefore $B_{5}=-B_{9}^{1}$. Choose $U,V\in \GL_{k}$ such that $UB_{5}V=\left[\begin{array}{cc}
0 & -I_{d}\\
0 & 0
\end{array}\right]$ for some $d$. Then 
\begin{align*}
B\sim & E_{1}(U,\diag(U,I_{n-3k}),V)B=\left[\begin{array}{cc|c|c}
I_{k} & 0 & 0 & 0\\
\hline 0 & UB_{4} & UB_{5}V & 0\\
\hline 0 & B_{7}^{1} & B_{8}^{1} & -UB_{5}V\\
0 & B_{7}^{2} & B_{8}^{2} & B_{9}^{2}V
\end{array}\right]\\
\sim & \left[\begin{array}{cccc|cc|cc}
I_{d} & 0 & 0 & 0 & 0 & 0 & 0 & 0\\
0 & I_{k-d} & 0 & 0 & 0 & 0 & 0 & 0\\
\hline 0 & 0 & 0 & 0 & 0 & -I_{d} & 0 & 0\\
0 & 0 & C_{1} & C_{2} & 0 & 0 & 0 & 0\\
\hline 0 & 0 & 0 & 0 & 0 & 0 & 0 & I_{d}\\
0 & 0 & C_{4} & C_{5} & C_{6} & 0 & 0 & 0\\
0 & 0 & C_{7} & C_{8} & C_{9} & 0 & 0 & 0\\
0 & 0 & 0 & 0 & 0 & 0 & I_{k-d} & 0
\end{array}\right]:=B'
\end{align*}

with $C_{2},C_{6}\in \Mat_{k-d}$. We apply Lemma \ref{lem:ind-sym-T} on $B'$ with:
\begin{align*}
S_{1}=&[d]\cup k+[d]\cup2k+[d],\,
S_{2}=[d]\cup n-k+1-[d]\cup n+1-[d]
,\,\\p=&\left[\begin{array}{cc}
 & I_{n-2k-d}\\
I_{d}
\end{array}\right]\end{align*} to get that it is enough to solve the problem for $d=0$. Therefore, we have
\[
B'=\left[\begin{array}{ccc|c|c}
I_{k} & 0 & 0 & 0 & 0\\
\hline 0 & C_{1} & C_{2} & 0 & 0\\
\hline 0 & C_{4} & C_{5} & C_{6} & 0\\
0 & C_{7} & C_{8} & C_{9} & 0\\
0 & 0 & 0 & 0 & I_{k}
\end{array}\right].
\]
For any $X\in \Mat_{k}$ we write $\beta_{C,X}=\left[\begin{array}{ccc}
-C_{5}X & C_{4} & C_{5}\end{array}\right]$. We note that: 
\[
t:=E_{2}(C_{2}X,\left[\begin{array}{c}
C_{5}\\
C_{8}\\
0
\end{array}\right]X)E_{5}(X C_{6},X\beta_{C,X})E_{1}(I_{k},\left[\begin{array}{ccc}
I\\
 & I\\
X &  & I
\end{array}\right],I_{k})
\]
stabilises $B'$ and has the character $f_{B'}^u(t)=\psi_{0}(\tr(X C_{6}-C_{2}X))$.
Therefore $C_{2}=C_{6}$. 

For $C\in X^{n-2k,k}$ denote $\rho_{C}=\left[\begin{array}{ccc}
I_{k}\\
 & C\\
 &  & I_{k}
\end{array}\right]\in X^{n,k}$. We define a map $\phi_{C}:F_{n-2k,k}\ra F_{n,k}$ by defining on
generators:
\begin{align*}
\phi_{C}(E_{1}(X,Y,Z)) & =E_{1}(X,\diag(X,Y,Z),Z)\\
\phi_{C}(E_{2}(a,b)) & =E_{3}(\left[\begin{array}{c}
a\\
b\\
0
\end{array}\right])\\
\phi_{C}(E_{3}(b)) & =E_{2}(C_{1}b,\left[\begin{array}{c}
C_{4}b\\
C_{7}b+bC_{4}b\\
0
\end{array}\right])E_{1}(I,\left[\begin{array}{ccc}
I\\
b & I\\
 &  & I
\end{array}\right],I)
\end{align*}

and similarly for $E_{4},E_{5}$, as we did in the definition of $\nu$
above. We extend to $F_{n-2k,k}$ as before by $\phi_{C}(w_{1}w_{2})=\phi_{w_{2}C}(w_{1})\phi_{C}(w_{2})$.
A simple calculation shows that $\rho_{w.C}=\phi_{C}(w)\rho_{C}$, $f_{\r_{C}}(\phi_{C}(w))=f_{C}(w)$ and $f_{\r_{C}}^u(\phi_{C}(w))=f_{C}^u(w)$ for any word $w\in F_{n-2k,k}$.

We know that $\rho_{C}$ is $*$-admissible and that $C_{2}=C_{6}$.
If $w.C=C$ for some $w\in F_{n-2k,k}$ then $\r_{C}=\phi_{C}(w)\r_{C}$
and $f_{C}^u(w)=f_{\r_{C}}^u(\phi_{C}(w))=1$. Therefore $C$ is $*$-admissible.
By the induction hypothesis there exists $w'\in F_{n-2k,k}$ such that $w'.C=\s(w')$ and $f_{C}(w')=1$. Then 
$$\s(\r_{C})=\r_{\s(C)}=\r_{w.C}=\phi_{C}(w')\s(\r_{C})$$ and $f_{\r_{C}}(\phi_{C}(w'))=1$,
showing that $\r_C$ is $*$-invariant. \qed

\appendix

\section{A gap in the proof of Lemma 3.9 in \texorpdfstring{\cite{Nien}}{[15]}}\label{app:nien}

In this appendix we use the notation in \cite{Nien}. We will
only state the gap and refer the reader to the relevant definitions
in Lemma 3.9 there. In the proof of Lemma 3.9, it is claimed that
\begin{align}
s_{2} \in S_{n}\label{eq:s2}
\end{align}
 if and only if 
\begin{align}\label{eq:weak}
\quad Q\left(\begin{array}{l}
Y_{1}\\
Y_{2}
\end{array}\right)=\left(\begin{array}{c}
R_{4}g_{3}C_{1}\\
p
\end{array}\right),\quad\left(Y_{3},Y_{4}\right)R^{-1}=\left(p,Q_{4}g_{1}D_{1}\right).
\end{align}

However, this is not necessarily true. A computation shows that Condition (\ref{eq:s2}) also implies
\begin{equation}
B_{4}r_{2}C_{2}=0.\label{eq:add-cond}
\end{equation}
As a concrete example of this issue, choosing:
\[
n=6,k=4,n'=1,R=\left(\begin{array}{cc|cc}
 &  &  & 1\\
 & 1\\
\hline 1 & \\
 &  & 1
\end{array}\right),Q=\left(\begin{array}{cc|cc}
 &  &  & 1\\
 & 1\\
\hline  &  & 1\\
-1 & 
\end{array}\right)
\]
and writing $g_{3}=\left(\begin{array}{cc}
r_{2}^{1} & r_{2}^{2}\\
0 & 0
\end{array}\right)$, we get that (\ref{eq:add-cond}) implies $r_{2}^{1}=0$. In contrast, (\ref{eq:weak}) has a solution for any $r_2^1\in \Mat_{n''}$.

To demonstrate how it affects the rest of the proof, later it is shown that for any
$g_{3}$ satisfying (\ref{eq:weak}) holds  $$\left(\begin{array}{cc}
r_{2}^{1} & r_{2}^{2}\\
0 & 0
\end{array}\right)\left(\begin{array}{cc}
V_{1}T_{2}-1\, & 0\\
V_{2}T_{2} & 0
\end{array}\right)=0$$ and then it is deduced there that $V_{1}T_{2}=1$. This deduction is not possible using (\ref{eq:s2}) since (\ref{eq:s2}) implies (\ref{eq:add-cond}) and in particular $r_{2}^{1}=0$.

{{
\bigskip
  \footnotesize
Itay Naor,
\textsc{Faculty of Mathematics and Computer Science, The Weizmann Institute of Science, POB
26, Rehovot 76100, ISRAEL}\par\nopagebreak \textit{E-mail address}: \texttt{itaynn@gmail.com}
}}
\end{document}